\documentclass[10pt]{amsart}
\usepackage[latin1]{inputenc}
\usepackage[T1]{fontenc}
\usepackage[english]{babel}
\usepackage[active]{srcltx}

\usepackage{colortbl}

\usepackage{amsmath,amssymb,amsthm,amsfonts}

\theoremstyle{plain}
\newtheorem{theorem}{Theorem}[section]
\newtheorem{corollary}[theorem]{Corollary}
\newtheorem{lemma}[theorem]{Lemma}
\newtheorem{cor}[theorem]{Corollary}

\theoremstyle{definition}
\newtheorem{remark}[theorem]{Remark}
\newtheorem{remarks}[theorem]{Remarks}

\numberwithin{equation}{section}

\newcommand{\N}{\mathbb{N}}
\newcommand{\nat}{\mathbb{N}}
\newcommand{\T}{\mathbb{T}}
\newcommand{\ein}{\boldsymbol{1}}

\newcommand{\eps}{\varepsilon}

\newcommand{\dopu}{{\,:}\allowbreak\ }

\DeclareMathOperator{\re}{Re}

\newcommand{\Id}{{\rm Id}}

\DeclareMathOperator{\supp}{supp}
\DeclareMathOperator{\dist}{dist}
\DeclareMathOperator{\cconv}{\overline{\mathrm{conv}}}

\newcommand{\caconv}{\overline{\mathrm{aconv}}}

\newcommand{\ext}[1][X^*]{\ensuremath{\mathrm{ext}(B_{#1})}}

\renewcommand{\leq}{\leqslant}
\renewcommand{\geq}{\geqslant}

\newcommand{\beq}{\begin{equation}}
\newcommand{\eeq}{\end{equation}}

\def\DP{Daugavet property}

\newcommand{\BS}{Banach space}

\begin{document}

\title[Lushness, numerical index 1 and the Daugavet property in r.i.\ spaces]
{Lushness, numerical index 1 and the Daugavet property in
  rearrangement invariant spaces}

\author[Kadets]{Vladimir Kadets}

\author[Mart\'{\i}n]{Miguel Mart\'{\i}n}

\author[Mer\'{\i}]{Javier Mer\'{\i}}

\author[Werner]{Dirk Werner}

\address[Kadets]{Department of Mechanics and Mathematics, Kharkov National
University,\linebreak
 pl.~Svobody~4,  61077~Kharkov, Ukraine}
\email{vova1kadets@yahoo.com}

\address[Mart\'{\i}n \& Mer\'{\i}]{Departamento de An\'{a}lisis Matematico\\
Facultad de Ciencias\\
Universidad de Granada\\
18071 Granada, Spain}
\email{\texttt{mmartins@ugr.es} \qquad \texttt{jmeri@ugr.es}}

\address[Werner]{Department of Mathematics, Freie Universit\"at Berlin,
Arnimallee~6, \qquad {}\linebreak D-14\,195~Berlin, Germany}
\email{werner@math.fu-berlin.de}

\thanks{The work of the first-named author
was supported by a grant from the {\it Alexander-von-Humboldt
Stiftung}. The second and third author were partially supported by
Spanish MICINN and FEDER project no.\ MTM2009-07498 and Junta de
Andaluc\'{\i}a and FEDER grants FQM-185 and
P09-FQM-4911.}

\date{February 14, 2011}

\keywords{lush space; numerical index; Daugavet property; K\"{o}the
  space; rearrangement invariant space}

\subjclass[2010]{Primary 46B04. Secondary  46E30.}

\begin{abstract}
We show that for spaces with 1-unconditional bases
lushness, the alternative Daugavet property and numerical
index~1 are equivalent. In the class of rearrangement
invariant (r.i.)\  sequence spaces the only examples of spaces with
these properties are $c_0$, $\ell_1$ and $\ell_\infty$.
The only lush r.i.\ separable function space on $[0,1]$ is $L_1[0,1]$;
the same space is the only r.i.\ separable function space on $[0,1]$
with the Daugavet property over the reals.
\end{abstract}


\maketitle

\section{Introduction}

The Daugavet property of a Banach space $X$ can be defined by
requiring that $\|\Id+T\|=1+\|T\|$ for all compact operators
$T\dopu X\to X$. (See Section~\ref{sec2} for a more
detailed discussion.) This is an isometric property of the
particular norm of $X$ and it is not
invariant under equivalent norms. In the setting of the classical
function spaces this property seems to be closely linked to the
sup-norm or the $L_1$-norm since for example $C[0,1]$,
$L_1[0,1]$, the disc algebra and $H^\infty$ (in their natural
norms) have the \DP\ whereas $L_p[0,1]$ fails it for $1<p<\infty$.
Nevertheless, there are very different examples of spaces with the
\DP, for instance, the space of Lipschitz functions on a metric
space, cf.\ \cite{IvKadWer}, which is in general not even an
$\mathcal{L}_\infty$-space; or some more exotic spaces such as
Talagrand's space \cite{KadSSW,Tal} and Bourgain-Rosenthal's space
\cite{BouRos,KW}. One of the main results in the present paper
(Corollary~\ref{cor-L1-Daug}) is that in the class of real separable
rearrangement invariant K\"othe function spaces on a finite
measure space there is, however, isometrically only one space with
the \DP, namely $L_1[0,1]$. (A somewhat weaker statement was
previously proved in \cite{AcKaMa}.)

We also study relatives of the \DP\ like the alternative \DP, lushness
and having numerical index~1. These properties will be recalled in
the next section. In Section~\ref{sec3} we prove, building on results
from \cite{AvKaMaMeSh2}, that these three properties are equivalent
for spaces with a 1-unconditional basis and that they characterise
$c_0$, $\ell_1$ or $\ell_\infty$ among the symmetric sequence spaces.

Let us briefly indicate the structure of the paper.
Section~\ref{sec2} contains pertinent definitions and background
material. In Section~\ref{sec3} we study symmetric sequence spaces
and prove the results just mentioned. Finally Section~\ref{sec4}
deals with lushness, the \DP\ and the almost \DP\ for rearrangement
invariant K\"othe function spaces.

We finish this introduction with some notation. We write $\T$ to
denote the set of (real or complex) scalars of modulus one. By
$\re(\cdot)$ we denote the real part if we are in the complex case
and just the identity if we are in the real case. Given a Banach
space $X$ and a subset $A\subset X$, we write $\cconv(A)$ to
denote the closed convex hull of $A$ and $\caconv(A)$ for the
closed absolutely  convex hull of $A$ (i.e.,
$\caconv(A)=\cconv(\T\,A)$). A \emph{slice} of a convex subset
$B\subset X$ is  a non-empty set which
is formed by the intersection of $B$ with an open real
half-space. Every slice of $B$ has the form
$$
S(B,x^*,\alpha):=\{x\in B\dopu \re x^*(x) > \sup \re x^*(B) -
\alpha \}
$$
for suitable $x^*\in X^*$ and $\alpha>0$.
Further, $B_X$ stands for
the closed unit ball of $X$ and $S_X$ for the unit sphere.

\section{Basic definitions}\label{sec2}
In this section we recall some basic definitions and facts
about real or complex rearrangement invariant spaces and the
properties we are going to investigate. For background on
rearrangement invariant spaces (and on K\"{o}the spaces in general)
we refer the reader to the classical book by J.~Lindenstrauss
and L.~Tzafriri \cite{LTII} for the real case, and to
\cite{OkadaRickerSanchezperez} for the complex case. In the
sequel we follow the notation of \cite{LTII}. Let $(\Omega,
\Sigma, \mu)$ be a complete $\sigma$-finite measure space. A
real or complex Banach space $X$ consisting of equivalence
classes, modulo equality almost everywhere, of locally
integrable scalar valued functions on $\Omega$ is a \emph{K\"{o}the
function space} if the following conditions hold.
\begin{itemize}
\item[(1)]
$X$ is \emph{solid}, i.e.,
if $|f|\leq|g|$ a.e.\  on $\Omega$ with $f$
    measurable and $g\in X$, then $f\in X$ and
    $\|f\|\leq\|g\|$.
\item[(2)] For every $A\in\Sigma$ with $\mu(A)<\infty$ the
    characteristic function $\ein_A$ of $A$ belongs to $X$.
\end{itemize}
Let us comment that the definition of a K\"{o}the space is usually
given in the real case (this is the case of \cite{LTII}), but it
extends to the complex case in an obvious way. Most of the basic
properties we are going to use are known in the real case
but their proofs extend without many problems to the complex
case.

If $X$ is a K\"{o}the function space, then every measurable
function $g$ on $\Omega$ so that $gf\in L_1(\mu)$ for every
$f\in X$ defines an element $x_g^*$ in $X^*$ by $x_g^*(f) =
\int_\Omega fg\,d\mu$. Any functional on $X$ of the form
$x_g^*$ is called an \emph{integral} and the linear space of
all integrals is denoted by $X'$. In the norm induced on $X'$
by $X^*$, this space is also a K\"{o}the function space on
$(\Omega, \Sigma, \mu)$.
The space $X$ is
\emph{order continuous} if whenever $\{f_n\}$ is a decreasing
sequence of positive functions which converges to~$0$ a.e.,
then $\{f_n\}$ converges to $0$ in norm.
(We note that for general Banach lattices, the above defines
$\sigma$-order continuity, which is weaker than order continuity in
this more general context.)
If $X$ is order
continuous, then every continuous linear functional on $X$ is
an integral, i.e., $X^*=X'$.

Let $(\Omega, \Sigma, \mu)$ be one of the measure spaces $\N$
or $[0,1]$. A K\"{o}the function space is a \emph{rearrangement
invariant} (\emph{r.i.\ space}) or \emph{symmetric space} if
the following conditions hold.
\begin{itemize}
\item[$(1)$] If $\tau\dopu  \Omega \to \Omega$ is a
    measure-preserving bijection and $f$ is an integrable
    function on $\Omega$, then $f\in X$ if and only if
    $f\tau^{-1}\in X$, and in this  case
    $\|f\|=\|f\tau^{-1}\|$.
\item[$(2)$] $X'$ is a norming subspace of $X^*$ and thus $X$ is
isometric to a subspace of $X''$. As a subspace of $X''$,
$X=X''$
or $X$ is the closed linear span of the simple integrable functions of $X''$.
\item[$(3)$] a. If $\Omega=\N$ then, as sets,
$$
\ell_1\subset X\subset\ell_\infty
$$
and the inclusion maps are of norm one, i.e., if $f\in \ell_1$
then $\|f\|_X\leq \|f\|_1$, and if $f\in X$ then $\|f\|_\infty\leq \|f\|_X$.

\noindent b. If $\Omega=[0,1]$ then, as sets,
$$
L_\infty[0,1]\subset X\subset L_1[0,1]
$$
and the inclusion maps are of norm one, i.e., if $f\in L_\infty[0,1]$
then $\|f\|_X\leq \|f\|_\infty$, and if $f\in X$ then $\|f\|_1\leq\|f\|_X$.
\end{itemize}

Let us emphasise some results on r.i.\ spaces which we will use
throughout the paper.

\begin{remarks}$ $\label{rem-norma-finite-sums}
\begin{itemize}
\item[(a)]
An r.i.\ space $X$ is order continuous
if and only if it is separable (cf.\ \cite[p.~118]{LTII}).
In this case, all bounded linear functionals on $X$ are integrals
(i.e., $X^*=X'$).
\item[(b)]
When $\Omega=\N$ we will denote
    $e_n=\ein_{\{n\}}\in X$ and $e_n'=\ein_{\{n\}}\in X'$
    for $n\in \N$. For $x\in X$, one has that
$$
\|x\|=\lim_{n}\Bigl\|\sum_{k=1}^n x_ke_k\Bigr\|.
$$
This can easily be deduced from the fact that $X'$ is norming
for $X$ and the monotone convergence theorem (see
\cite[Proposition~1.b.18]{LTII}).
\end{itemize}
\end{remarks}

We now discuss the isometrical Banach space properties in which we are
interested in this paper.

A \BS\ $X$  has the \textit{\DP} if the following identity
 \beq \label{0-eq1}
  \|\Id+T\|=1+\|T\|,
 \eeq
called {\it the Daugavet equation}, holds true for every rank-one
operator $T\dopu  X \to X$, i.e., $T = f \otimes e$, where $e
\in X$ and $f \in X^*$.
This notion was introduced in \cite{KadSSW} where it was shown that
then weakly compact operators also satisfy~(\ref{0-eq1}). It was also
shown in \cite[Lemma~2.2]{KadSSW} that this property is equivalent to the
following slice condition:
\begin{quote}
For every $x \in S_X$, $x^* \in S_{X^*}$, and $\eps > 0$
there is a $y \in S_X$ such that $\re x^*(y) > 1 - \eps$ and $\|x +
y\| > 2 - \eps$.
\end{quote}

A weakening of the \DP\ was introduced in  \cite{MaOi}.
If every rank-one operator $T\in L(X)$
satisfies the norm equality
\begin{equation}\label{aDE} 
\max_{\theta\in\T}\|\Id + \theta\,T\|=1 + \|T\|,
\end{equation}
 $X$ is
said to have the \emph{alternative Daugavet property} ({ADP} for
short). In this case again all weakly compact operators on $X$
also satisfy \eqref{aDE}. A slice characterisation similar to the
above one holds for the ADP as well \cite[Proposition~2.1]{MaOi}:
\begin{quote}
For every $x \in S_X$, $x^* \in S_{X^*}$, and $\eps > 0$ there is
a $y \in S_X$ such that $\re x^*(y) > 1 - \eps$ and $\max_{\theta
  \in \T}\|\theta x + y\| > 2 - \eps$.
\end{quote}

This notion is strongly linked to the theory of numerical ranges.
For every $T \in L(X)$ the quantity
$$
v(T)=\sup\bigl\{|x^*(Tx)| \dopu   x\in S_X,\ x^*\in S_{X^*},\
x^*(x)=1\bigr\}
$$
is called the \emph{numerical radius} of $T$.
 A Banach space is said to have \emph{numerical index~$1$}
\cite{D-Mc-P-W} if every $T\in L(X)$ satisfies the condition
$v(T)=\|T\|$.  It is known \cite{D-Mc-P-W} that
$$
v(T)=\|T\| \qquad \Longleftrightarrow \qquad T \text{ satisfies \eqref{aDE}.}
$$
Thus, $X$ has numerical index~1 if and only if every $T\in L(X)$
satisfies \eqref{aDE}. Evidently, both the Daugavet property and
numerical index~1 imply the ADP. On the other hand, the space
$C([0,1], \ell_2)\oplus_\infty c_0$ has the ADP, but it has neither
the Daugavet property nor  numerical index~1
\cite[Example~3.2]{MaOi}.

A Banach space $X$ is said to be \emph{lush} \cite{BKMW} if for
every $x,y\in S_X$ and every $\eps>0$, there is $x^*\in
S_{X^*}$ such that
$$
x\in S:=S(B_X,x^*,\eps) \qquad \text{and} \qquad
\dist (y,\caconv (S) ) < \eps.
$$
Lush spaces have numerical index~1
\cite[Proposition~2.2]{BKMW}, but it has very recently been
shown that the converse result is not true in general
\cite[Remark~4.2.a]{KMMS}.


\section{Sequence spaces}\label{sec3}
In this section we demonstrate that in the class of  r.i.\  sequence spaces
ADP, lushness and numerical index~1 are equivalent properties,
and that apart from
the classical examples  $c_0$, $\ell_1$ and $\ell_\infty$ there are no
r.i.\  sequence spaces with these properties. We remark that r.i.\
sequence spaces do not have the \DP\ since they admit
rank-one unconditional projections.


\subsection{ADP, lushness and numerical index~1
 are equivalent for spaces with 1-unconditional bases}

Let $X$ be a Banach space and let $A$ be a convex bounded subset of
$X$. According to \cite{AvKaMaMeSh1}, \cite{AvKaMaMeSh2} (see also
\cite{KaSh2}) a
countable family $\{V_n \dopu  n\in\nat\}$ of subsets of $A$
is called \emph{determining} for $A$ if $A \subset \cconv (B)$ for
every $B\subset A$ intersecting all the sets $V_n$. Equivalently,
$\{V_n \dopu  n\in\nat\}$ is determining for $A$ if every slice
of $A$ contains one of the $V_n$.

A convex bounded subset $A$ of a Banach space $X$ is
\emph{slicely countably determined} (\emph{SCD set} for short)
if there is a determining sequence of slices of $A$.
Equivalently \cite[Proposition 2.18]{AvKaMaMeSh2},
$A$ is SCD if there is a determining sequence of
relatively weakly open subsets of $A$.

The next theorem gives in particular a positive answer to
 Question~7.4(b) of~\cite{AvKaMaMeSh2}.

\begin{theorem}
Let $X$ be a space with a $1$-unconditional basis $(e_n)_{n \in \N}$.
Then $B_X$ is an SCD set.
\end{theorem}

\begin{proof}
Fix a countable dense subset $D \subset B_X$ consisting of vectors
with finite supports, and for every $a \in D$, $a = \sum_{k=1}^n a_ke_k$, select
the corresponding relative weak neighbourhoods
$$
U(a, m) = \biggl\{x=\sum_{k=1}^\infty x_ke_k \in B_X\dopu  \max_{j \leq n}|a_j - x_j|
 < \frac1m \biggr\}.
$$
Let us show that $\{U(a, m)\dopu  a \in D, \ m \in \N\}$ is a determining
collection
of weak neighbourhoods. Let $V \subset B_X$ be a  closed convex set that
intersects all the $U(a, m)$.
For fixed $f \in S_{X^*}$ and $\eps > 0$, we have to show that the slice
$$
S(B_X,f,\eps)= \{u \in B_X\dopu \re f(u) > 1 - \eps\}
$$
intersects $V$. To do so, take $a = \sum_{k=1}^n a_ke_k \in D \cap
S(B_X,f,\eps/4)$
and observe that there is an element $x=\sum_{k=1}^\infty x_ke_k \in V$
whose first $n$ coordinates are arbitrarily close to the corresponding $a_k$
so that we can assume $\|a - \sum_{k=1}^n x_ke_k\| < \eps/4$. Therefore, we have
$$
\re f \Bigl(\sum_{k=1}^n x_ke_k \Bigr)>\re f(a)-
\frac{\eps}{4}>1-\frac{\eps}{2}\,.
$$
Besides, by 1-unconditionality it is clear that
$\|\sum_{k=1}^n x_ke_k - \sum_{k=n+1}^\infty x_k e_k\| = \|x\| \leq 1$.
Hence, we can finally write
$$
\re f(x)=2\re f\Bigl(\sum_{k=1}^n x_ke_k\Bigr)-
\re f\Bigl(\sum_{k=1}^n x_ke_k - \sum_{k=n+1}^\infty x_k e_k \Bigr) > 1-\eps
$$
which gives $x\in V\cap S(B_X,f,\eps)$, finishing the proof.
\end{proof}

It is known that every Banach space $X$ with the
alternative Daugavet property
whose unit ball is an SCD set is lush
\cite[Theorem 4.4]{AvKaMaMeSh2}, so we have the following corollary.

\begin{cor} \label{cor1}
In the class of spaces with  $1$-unconditional bases the three properties
ADP, lushness and numerical index~$1$ are equivalent.
\end{cor}


\subsection{The only separable r.i.\ sequence spaces with numerical
  index~$\boldsymbol1$ are $\boldsymbol{c_0}$ and $\boldsymbol{\ell_1}$}
First note that separable r.i.\ sequence spaces are nothing but Banach
spaces with 1-symmetric bases. So the results of the previous subsection
are applicable to this kind of spaces. We start with the separable case.

\begin{theorem}\label{thm-sequences}
Let $X$ be a separable r.i.\ space on $\N$. If $X$
is lush, then $X$ is $c_0$ or $\ell_1$.
\end{theorem}

For the proof of this result we need two easy lemmas. The second one, which
we state here for the readers' convenience, appears in \cite{LeMaMe}
for the wider class of Banach spaces with $1$-unconditional bases.

\begin{lemma}\label{lemma-isometries}
Let $X$ be a Banach space and let $x^*\in S_{X^*}$ be such that
$|x^{**}(x^*)|=1$ for every $x^{**}\in \ext[X^{**}]$. If
$J\dopu X^*\to X^*$ is an onto isometry then
$|x^{**}(Jx^*)|=1$ for every $x^{**}\in \ext[X^{**}]$.
\end{lemma}

\begin{proof}
Just note that $J^*$ is an onto isometry on $X^{**}$ and
thus $J^*(x^{**})\in \ext[X^{**}]$ for every $x^{**}\in \ext[X^{**}]$.
\end{proof}

\begin{lemma}\cite[Lemma~3.2]{LeMaMe}\label{lemma-coordinates} Let
$X$ be an r.i.\ space on $\N$ and let $x^*\in S_{X^*}$ be such that
$|x^{**}(x^*)|=1$ for every $x^{**}\in \ext[X^{**}]$. Then,
$$
|x^*(n)|\in\{0,1\} \qquad \text{for every } n\in\N.
$$
\end{lemma}

\begin{proof}[Proof of Theorem~\ref{thm-sequences}]
Since $X$ is a separable lush space, Theorem~4.3 in \cite{KMMP} tells us that the set
$$
\mathcal{A}=\{x^*\in S_{X^*} \dopu  |x^{**}(x^*)|=1 \text{ for every } x^{**}\in\ext[X^{**}]\}
$$
is norming for $X$. By Lemma~\ref{lemma-coordinates} one has
$|a^*(n)|\in\{0,1\}$ for every $a^*\in \mathcal{A}$ and every $n\in \N$.
Therefore, we can split the proof into the following two cases:

$(1)$ \emph{There is an $a^*_0\in \mathcal{A}$ such that the set
$\text{I}=\{n\in\N \dopu |a^*_0(n)|=1\}$ is infinite}.
In this case $X$ is isometrically isomorphic to $\ell_1$. Indeed, for
fixed $x\in B_X$ and $N\in\N$, consider $\omega_n\in\T$  such that
$\omega_nx(n)=|x(n)|$ for every $n=1,\dots,N$ and define
$x^*=\sum_{n=1}^N \omega_ne_n'$. Next, take a bijection
$\tau\dopu \N\to \N$ such that $\tau(\{1,\dots,N\})\subset I$
and consider the onto isometry $J\in L(X)$ given by $Jx=x\tau^{-1}$ ($x\in X$).
Then, $J^*\in L(X^*)$ is an onto isometry such that
$$
\{1,\dots,N\}\subset\{n\in\N \dopu  |(J^*a^*_0)(n)|=1 \}.
$$
Therefore, we can write
$$
1\leq\|x^*\|\leq\|J^*a^*_0\|=1
$$
and so $\|x^*\|=1$. Finally, it suffices to observe that
$$
\sum_{n=1}^N|x(n)|=|x^*(x)|\leq \|x\|
$$
which tells us that $x\in \ell_1$ and $\|x\|_1\leq\|x\|$
(the reversed inequality is always true).

$(2)$ \emph{For every $a^*\in \mathcal{A}$ the set $\{n\in\N \dopu
|a^*(n)|=1\}$ is finite.}
In this case we will show that $X$ is isometrically isomorphic to
$c_0$. For
fixed $a^*\in \mathcal{A}$ we are first going to show that
$\#\supp(a^*)=1$. Using
Lemma~\ref{lemma-isometries} we can assume, up to isometry, that
$a^*(n)=1$ for every $n\in\supp(a^*)$. Take $x^{**}\in \ext[X^{**}]$
satisfying $x^{**}(a^*)=1$ and observe that $x^{**}(e_n')\geq0$ for
every $n\in \supp(a^*)$. We claim that $\{x^{**}(e_n')\}_{n\in\N}$ is constant
on the support of $a^*$. Indeed, fix $j,k\in \supp(a^*)$ and
$m\notin \supp(a^*)$, take $\omega\in \T$ satisfying $\omega
x^{**}(e_m')=|x^{**}(e_m')|$, and define $a^*_j,a^*_k\in \mathcal{A}$ by
$$
a^*_j=\omega e_m'+\sum_{n\in\supp(a^*)\setminus\{j\}}e_n' \qquad
\text{and} \qquad a^*_k=\omega
e_m'+\sum_{n\in\supp(a^*)\setminus\{k\}}e_n'
$$
(observe that $a^*_j, a^*_k\in \mathcal{A}$ by Lemma~\ref{lemma-isometries}).
Then we can write
\begin{align*}
|x^{**}(e_m')|+\sum_{n\in\supp(a^*)\setminus\{j\}}x^{**}(e_n')
&=|x^{**}(a^*_j)|
=1
=|x^{**}(a^*_k)| \\
&=|x^{**}(e_m')|+\sum_{n\in\supp(a^*)\setminus\{k\}}x^{**}(e_n')
\end{align*}
and, therefore, $x^{**}(e_j')=x^{**}(e_k')$. So $x^{**}(e_n')=\frac{1}{\#\supp(a^*)}$ for every
$n\in\supp(a^*)$. Now it is clear that $\#\supp(a^*)=1$:  otherwise
there are $j\neq k$ in $\supp(a^*)$; we define $\tilde a^*\in \mathcal{A}$ by
$$
\tilde a^*=e_k'-\sum_{n\in \supp(a^*)\setminus\{k\}}e_n'
$$
and we observe that
$$
1=|x^{**}(\tilde
a^*)|=\Bigl|x^{**}(e_k')-\sum_{n\in\supp(a^*)
\setminus\{k\}}x^{**}(e_n')\Bigr|=\frac{\#\supp(a^*)-2}{\#\supp(a^*)},
$$
which is impossible.

Finally, since $\mathcal{A}$ is norming for $X$ we have that
$$
B_{X^*}=\caconv^{\,w^*}(\mathcal{A})=\caconv^{\,w^*}(\{e_n'\dopu  n\in\N\})
$$
and thus, $\|x\|=\sup\{|x(n)|\dopu n\in\N\}$ for every $x\in X$. Since $X$
is the closed
linear span of $\{e_n\dopu  n\in \N\}$ we deduce that $X$ is isometric to
$c_0$, finishing the proof.
\end{proof}

The last theorem together with Corollary~\ref{cor1} gives
the result announced in the title of the  subsection.

\begin{cor}\label{cor2}
The only separable r.i.\  sequence spaces with numerical index~$1$ are
$c_0$ and $\ell_1$.
The same spaces are the only examples of separable r.i.\  sequence
spaces with the ADP.
\end{cor}


\subsection{The only non-separable r.i.\  sequence space with
  numerical index 1 is  $\boldsymbol{\ell_\infty}$}

\begin{theorem}\label{thm-non-sepADP}
Let $X$ be a non-separable r.i.\  space on $\N$. If $X$ has the ADP, then
$X$ is $\ell_\infty$.
\end{theorem}

We need a preliminary result whose proof is borrowed from \cite[Theorem 1.1]{AcKaMa}.

\begin{lemma}\label{lemm-subspace}
Let $X$ be an r.i.\  space on $\N$. Denote by $E$ the closed linear span
of the set of canonical
basis vectors $e_n$, $n \in \N$. If $X$ has the ADP, then
$E$ also has the ADP.
\end{lemma}

\begin{proof}
Fix $x \in S_E$, $f \in S_{E^*}$  and $\eps > 0$. Our goal is to
find a $y \in B_E$ with
\begin{equation} \label{a1}
|f(y)| > 1 - \eps \quad {\textrm{and}} \quad
\max_{\theta\in\T}\|\theta x + y\| > 2 - \eps.
\end{equation}
First remark that $f$ can be considered as a sequence of scalars
$f = (f_1,f_2, \ldots )$ that acts on arbitrary $z = (z_1, z_2, \ldots)$
by
\begin{equation} \label{a2}
f(z) = \sum_{k=1}^\infty f_k z_k.
\end{equation}
By the same formula (\ref{a2}), $f$ defines a linear functional on $X$
with $\|f\|_{X^*} = \|f\|_{E^*} = 1$. Since $X$ has the ADP, there is a
 $z = (z_1, z_2, \ldots) \in B_X$ with
$$
|f(z)| > 1 - \eps \quad {\textrm{and}}
\quad \max_{\theta\in\T}\|\theta x + z\| > 2 - \eps.
$$
One may now select $n \in \N$ big enough to fulfill
$\max_{\theta\in\T}\|\sum_{k=1}^n z_k e_k  + \theta x\|> 2 - \eps$ and
$|\sum_{k=1}^n z_k f_k| >  1 - \eps $. Then $y := \sum_{k=1}^n z_k e_k \in E$
fulfills the conditions~(\ref{a1}).
\end{proof}

\begin{proof}[Proof of Theorem~\ref{thm-non-sepADP}]
According to Lemma~\ref{lemm-subspace}, the subspace $E \subset
X$ spanned by the canonical basis vectors $e_n$ has the ADP.
Since $E$ is a separable  r.i.\  space on $\N$, $E$ must be
either $c_0$ or $\ell_1$ by Corollary~\ref{cor2}. When
$E=\ell_1$, for fixed $x\in X$ one has that
$$
\|x\|=\lim_{n}\Bigl\|\sum_{k=1}^n x_ke_k\Bigr\|=\lim_{n}\sum_{k=1}^n|x_k|
$$
by Remark~\ref{rem-norma-finite-sums}.b, and so $X\subset
\ell_1$ contradicting the non-separability of $X$.

When $E = c_0$, we fix $x\in X$ and use again
Remark~\ref{rem-norma-finite-sums}.b to deduce that
$$
\|x\|= \sup \{|x_n|\dopu n\in \N\}
$$
and then $X\subset \ell_\infty$ isometrically. So it remains to
show that every element of $\ell_\infty$ lies in $X$. Since $X$
is solid, it suffices to check that the element
$(1,1,1,\ldots)$ lies in $X$. Let $x=(x_1,x_2,\ldots)\in
X\setminus c_0$ be such that $x_n\geq0$ for every $n\in \N$ and
$\limsup x_n>1$. Using this and rearranging $x$ we may suppose
without loss of generality that $x_{2n}>1$ for every $n\in \N$.
Therefore, using again the solidity, we get that
$(0,1,0,1,\ldots)\in X$ and, by symmetry, that
$(1,0,1,0,\ldots) \in X$. Finally, we can write
$$
(1,1,1,1,\ldots)=(0,1,0,1,\ldots)+(1,0,1,0,\ldots) \in X
$$
which finishes the proof.
\end{proof}

\section{Symmetric function spaces  on [0,1]}
\label{sec4}

\subsection{Lushness in separable r.i.\   function spaces}

Our next goal is to prove a result similar to Theorem~\ref{thm-sequences}
for rearrangement
invariant function spaces on $[0,1]$. To do so, we need the following
easy lemma.

\begin{lemma}\label{lemma-positive-functional}
Let $X$ be a K\"{o}the function space and $f\in S_X$ with $f \geq 0$. If
$x^*\in S_{X^*}$ satisfies $x^*(f)=1$ then $x^*$ is positive on the
subspace $X_f=\{g\in X\dopu   |g|\leq cf $ for some $c>0\}$.
\end{lemma}

\begin{proof}
Let $0\leq g\in X_f$; multiplying by a suitable constant we can assume
without loss of generality that $0\leq g\leq f$. Therefore, we have
for all scalars $\theta\in\T$
$$
|g+\theta(f-g)|\leq |g| + |f-g| = f
$$
and consequently
$\|g+\theta(f-g)\|\leq \|f\|=1$. Therefore
$$
|x^*(g) + \theta (1-x^*(g))|\leq1
$$
 for all $\theta\in\T$,
and it follows that
$$
|x^*(g)| + |1-x^*(g)|\leq1
$$
which means that $x^*(g)$ is a real number in the interval $[0,1]$.
\end{proof}

We can now present the promised result.

\begin{theorem}
Let $X$ be a separable r.i.\ space on $[0,1]$. If
$X$ is lush, then $X=L_1[0,1]$.
\end{theorem}

\begin{proof}
Since $X$ is a separable lush space, Theorem~4.3 in \cite{KMMP}
tells us that the set
$$
\mathcal{A}=\{f\in S_{X^*} \dopu  |x^{**}(f)|=1 \text{ for every }
x^{**}\in\ext[X^{**}]\}
$$
is norming for $X$. To finish the proof it suffices to show
that $|f|=\ein$ for every $f\in \mathcal{A}$ (recall that
$X^*=X'$ by separability). Indeed, given $x\in X$ we can then write
$$
\|x\|_1\leq \|x\|\leq
\sup\left\{\Bigl|\int_0^1f(t)x(t)\,dt \Bigr| \dopu   f\in
\mathcal{A}\right\} \leq  \|x\|_1,
$$
and taking into account that all simple functions are in $X$ we  obtain
that $X=L_1[0,1]$.

For fixed $f\in \mathcal{A}$, there is an onto isometry on $X^*$ sending $f$ to $|f|$.
Then Lemma~\ref{lemma-isometries} tells us that $|f|\in \mathcal{A}$ and so we can
assume without
loss of generality that $f\geq0$. Since $f\in X\subset L_1[0,1]$, there exist
positive numbers $\alpha$ and $\Delta$ such that
$$
\mu\left(\{t\in [0,1] \dopu  \alpha\leq f(t)\}\right)\geq\Delta.
$$

\noindent\textbf{Claim:} For every $x^{**}\in \ext[X^{**}]$ and
every $\delta>0$ there is an interval $I\subset[0,1]$ with
$\mu(I)\leq \delta$ such that
$x^{**}(g\boldsymbol{1}_{[0,1]\setminus I})=0$ for every $g\in
L_\infty[0,1]$.

\noindent
\emph{Proof of the Claim.} We may suppose that $0<\delta<\Delta/2$.
Now, consider a partition of $[0,1]=I_1\cup\cdots \cup I_n$ into
disjoint intervals with
$\mu(I_k)\leq \delta$ for $k=1,\dots,n$ and find for
 fixed $j,k\in\{1,\dots,n\}$ a rearrangement $\tilde f$ of $f$ such that
\begin{equation}\label{eq:thm-functions-ftilde}
\alpha\boldsymbol{1}_{I_j}\leq\tilde f\boldsymbol{1}_{I_j}\qquad \text{and}\qquad
\alpha\boldsymbol{1}_{I_k}\leq\tilde f\boldsymbol{1}_{I_k}.
\end{equation}
Take $\omega\in\T$ such that $\omega x^{**}(\tilde f)=1$ and use
Lemma~\ref{lemma-positive-functional} to obtain that the functional
$\omega x^{**}$ is positive on the subspace $X_{\tilde f}=\{g\in
X\dopu  |g|\leq c\tilde f$ for some $c>0\}$.
Now observe that
$$
\tilde f_\theta=\tilde f\boldsymbol{1}_{I_j}+\theta\tilde
f\boldsymbol{1}_{[0,1]\setminus I_j}\in \mathcal{A}
$$
for every $\theta \in \T$. Therefore, there are $\theta_1,\theta_2 \in
\T$ satisfying
\begin{align*}
1&=|x^{**}(\tilde f_{\theta_1} )|=|x^{**}(\tilde
f\boldsymbol{1}_{I_j})|+|x^{**}(\tilde
f\boldsymbol{1}_{[0,1]\setminus I_j})|\qquad \text{and}\\
1&=|x^{**}(\tilde f_{\theta_2} )|=\Big||x^{**}(\tilde
f\boldsymbol{1}_{I_j})|-|x^{**}(\tilde
f\boldsymbol{1}_{[0,1]\setminus I_j})|\Big|
\end{align*}
which clearly implies $\big\{|x^{**}(\tilde
f\boldsymbol{1}_{I_j})|, |x^{**}(\tilde
f\boldsymbol{1}_{[0,1]\setminus I_j})|\big\}=\{0,1\}$. Suppose first that
$|x^{**}(\tilde f\boldsymbol{1}_{I_j})|=0$ and use
\eqref{eq:thm-functions-ftilde}
to observe that $g\boldsymbol{1}_{I_j}\in X_{\tilde f}$ for every
$g\in L_\infty[0,1]$.
 Thus we can write
$$
|\omega x^{**}(g\boldsymbol{1}_{I_j})|\leq
\omega x^{**}(|g|\boldsymbol{1}_{I_j})
\leq \omega x^{**}(\tilde f\boldsymbol{1}_{I_j})=0
$$
and, therefore, $x^{**}(g\boldsymbol{1}_{I_j})=0$. If
otherwise $|x^{**}(\tilde f\boldsymbol{1}_{[0,1]\setminus I_j})|=0$, then
$$
|\omega x^{**}(g\boldsymbol{1}_{I_k})|\leq
\omega x^{**}(|g|\boldsymbol{1}_{I_k})
\leq \omega x^{**}(\tilde f\boldsymbol{1}_{[0,1]\setminus I_j})=0.
$$
Hence we have shown that either $x^{**}(g\boldsymbol{1}_{I_j})=0$ for every
$g\in L_\infty[0,1]$ or $x^{**}(g\boldsymbol{1}_{I_k})=0$ for
every $g\in L_\infty[0,1]$.
Finally, the arbitrariness of $j,k\in \{1,\dots,n\}$ finishes
the proof of the Claim.

We continue the proof showing that
 $f$ is necessarily bounded (see (a) below) and even constant (see (b) below).

 (a) Suppose for contradiction that $f$ is unbounded and define
$$
B=\{t\in [0,1]\dopu f(t)\geq 6\}
$$
which has positive measure and satisfies $\mu(B)\leq1/6$. Fix
$x^{**}\in\ext[X^{**}]$ and use the Claim to find an interval
$I\subset[0,1]$ with $\mu(I)\leq \mu(B)$ and such that
$x^{**}(g\boldsymbol{1}_{[0,1]\setminus I})=0$ for every $g\in
L_\infty[0,1]$. Up to rearrangement of $f$, we can assume without loss
of generality that $I\subset B$. For every $n\in \N$ consider the set
$$
B_n=\{t\in [0,1]\dopu 5+n\leq f(t)<6+n\},
$$
split it into two sets of equal measure $B_n=B_{n,1}\cup B_{n,2}$, and define
$$
B_1=\bigcup_{n=1}^\infty B_{n,1}\,, \qquad  B_2=
\bigcup_{n=1}^\infty B_{n,2}\,,\qquad f_1=f\boldsymbol{1}_{B_1}\,,
\qquad \text{and} \qquad f_2=f\boldsymbol{1}_{B_2}\,.
$$
For $\{i,j\}=\{1,2\}$, take an automorphism $\tau$ of $[0,1]$
which fixes $[0,1]\setminus B$ and sends $B_{n,i}$ to $B_{n,j}$
for every $n\in\N$. Calling $\hat f_i=f_i \tau^{-1}$, observe
that
\begin{equation}\label{eq:thm-functions-prop-fi}
f_j\leq \frac76\,\hat f_i.
\end{equation}
Besides, we can write
$$
1=|x^{**}(f)|=|x^{**}(f_1)+x^{**}(f_2)+x^{**}(f\boldsymbol{1}_{[0,1]\setminus B})|
=|x^{**}(f_1)+x^{**}(f_2)|
$$
and using an argument as in the proof of the Claim one can easily
deduce the existence of $j\in \{1,2\}$ such that
$|x^{**}(f_j)|=1$. Since
$\mu(\{t\in[0,1]\dopu \tilde f(t)\leq2\}) \geq \frac12$ we can now take
 a rearrangement $\tilde f$ of $f$ such that
$$
\tilde B\cap B=\emptyset \qquad\text{and} \qquad
I\subset\{t\in[0,1]\dopu \tilde f(t)\leq2\}
$$
where $\tilde B=\{t\in [0,1] \dopu \tilde f(t)\geq 6\}$. Using this
and $I\subset B$,
observe that $|\tilde f\boldsymbol{1}_I|\leq \frac26 |f\boldsymbol{1}_I|$ and,
therefore, $\|\tilde f\boldsymbol{1}_I\|\leq\frac13$. Let $\tilde B_i$ and
$\tilde f_i$ be the corresponding rearrangements of $B_i$ and $f_i$ associated
to $\tilde f$. Next use that $\tilde f
\boldsymbol{1}_{[0,1]\setminus\tilde B}\in L_\infty[0,1]$
to deduce that $x^{**}(\tilde f\boldsymbol{1}_{[0,1]\setminus (\tilde B \cup I)})=0$.
Hence, one can write
\begin{align*}
1=|x^{**}(\tilde f)|&=|x^{**}(\tilde f_1)+x^{**}(\tilde f_2)
+x^{**}(\tilde f\boldsymbol{1}_I)+x^{**}(\tilde
f\boldsymbol{1}_{[0,1]\setminus (\tilde B \cup I)})|\\
&\leq|x^{**}(\tilde f_1)|+|x^{**}(\tilde f_2)|+\frac13\,.
\end{align*}
Therefore, there is $k\in \{1,2\}$ such that $|x^{**}(\tilde f_k)|\geq\frac13$.

Finally, for each $\theta\in \T$ define $g_\theta=f_j+\theta\tilde
f_k$ and take $\theta_0\in\T$ such that
$$
\frac43\leq|x^{**}(f_j)|+|x^{**}(\tilde
f_k)|=|x^{**}(g_{\theta_0})|\leq\|g_{\theta_0}\|.
$$
If $j\neq k$ then one obviously has $\|g_{\theta_0}\|\leq\|f\|=1$,
which is a contradiction.
If otherwise $j=k$, use \eqref{eq:thm-functions-prop-fi} to deduce that
$$
|g_{\theta_0}|\leq f_j+\tilde f_k \leq\frac76(\hat f_i+\tilde f_k)
$$
where $i\in\{1,2\}\setminus\{j\}$. Therefore,
$$
\|g_{\theta_0}\|\leq\frac76 \|\hat f_i+\tilde f_k\|\leq\frac76 \|f\|<\frac43,
$$
which again gives us a contradiction. Hence, we have that
$f\in L_\infty[0,1]$.

(b) Suppose for contradiction that $f$ is non-constant.
Then, there are numbers $0\leq c<d$ and sets
$$
C=\{t\in [0,1]  \dopu   f(t)\leq c \}
\qquad \text{and} \qquad D=\{t\in [0,1]  \dopu   f(t)\geq d \}
$$
such that $\mu(C)>0$ and $\mu(D)>0$. Now fix $x^{**}\in \ext[X^{**}]$ and use the
Claim to find an interval
$I\subset[0,1]$ with $\mu(I)\leq\min \{\mu(C), \mu (D)\}$
and such that $x^{**}(g\boldsymbol{1}_{[0,1]\setminus I})=0$ for
every $g\in L_\infty [0,1]$.
Take rearrangements $f_1,f_2\in \mathcal{A}$  of $f$ such that
$$
f_1(t)\leq c \qquad\text{and} \qquad f_2(t)\geq d \qquad \text{for every }
 t\in I;
$$
hence $\|f_1\boldsymbol{1}_{I}\| \leq\frac{c}{d}\|f_2\boldsymbol{1}_{I}\|$.
Then, using the fact that $f_1\in L_\infty[0,1]$, we can write
$$
1=|x^{**}(f_1)|=
|x^{**}(f_1\boldsymbol{1}_{I})+x^{**}(f_1\boldsymbol{1}_{[0,1]\setminus I})|=
|x^{**}(f_1\boldsymbol{1}_{I})|
$$
and, therefore,
$$
1\leq\|f_1\boldsymbol{1}_{I}\|\leq
\frac{c}{d}\|f_2\boldsymbol{1}_{I}\|\leq\frac{c}{d}<1,
$$
which is the desired contradiction.

Hence, $f\in L_\infty[0,1]$ and it is constant. Now it is immediate to
deduce that $f=\ein$ finishing the proof.
\end{proof}


\subsection{The Daugavet property in separable r.i.\ function spaces}

In their
paper \cite{AcKaMa}, M.~D.~Acosta, A.~Kami{\'n}ska, and
M.~Masty{\l}o proved in Proposition~1.6 that if a separable real r.i.\
function space on $[0,1]$ with the Fatou property
has the \DP, then $X$ as a set of
functions coincides with $L_1[0,1]$, but they left open the
question whether the norm on $X$ is necessarily the same as the
standard $L_1$-norm.

In this section we answer the above question in the positive even if
we remove the assumption of the Fatou property;
i.e., we show that the only separable real r.i.\  function space on
$[0,1]$ with the \DP\ is $L_1[0,1]$ endowed with its canonical
norm.

Below $X$ is a separable (hence order continuous) real r.i.\ function
space on $[0,1]$. We remark that order
continuity implies that the subspace of simple functions and the
subspace of continuous functions are dense in $X$. Denote by
$\phi$ the \textit{fundamental function} of $X$, that is $\phi(t)
= \|\ein_{[0,t]}\|_X$. Let us list here some known properties
of~$\phi$:
\begin{enumerate}
\item[(a)] $\phi$ is non-decreasing,
\item[(b)] $t \leq \phi(t)\leq 1$,
\item[(c)] $\phi(t + \tau) \leq \phi(t) + \phi(\tau)$,
\item[(d)] $\lim_{t \to 0}\phi(t)=0$ (see \cite[Chapter 2,
    Theorem~5.5]{BenSha}, for instance).
\end{enumerate}

We need several preliminary results.
The first one is certainly
known, but we haven't been able to locate a reference. It
characterises
$L_1[0,1]$ among separable r.i.\ function spaces on $[0,1]$.

\begin{lemma}\label{lemma-L1-char}
Let $X$ be a separable r.i.\ function space on $[0,1]$
and let $\phi$ be its
fundamental function. If $\liminf_{\tau \to 0} \phi(\tau)/\tau =
1$, then $X = L_1[0,1]$ endowed with its canonical norm.
\end{lemma}

\begin{proof}
It is sufficient to prove that $\phi(t) = t$ for all $t \in [0,1)$.
Indeed, in this case for every simple function $f =
\sum_{k=1}^na_k\ein_{A_k}$ we have that $\|f\|_{L_1} \leq \|f\|_{X} \leq
\sum_{k=1}^n |a_k| \phi(\mu(A_k)) = \|f\|_{L_1}$. So fix $t \in [0,1)$
and select a sequence of $\tau_n
> 0$, $\tau_n \to 0$, such that $\phi(\tau_n)/\tau_n \to 1$.
Denote $m(n)$ the smallest positive integer such that
$m(n)\tau_n \geq t$ and observe that $t\leq m(n)\tau_n < t +
\tau_n$. Then
$$
t \leq \phi(t) \leq \phi(m(n)\tau_n)  \leq m(n) \phi(\tau_n) =
 \tau_n m(n) \phi(\tau_n)/\tau_n \to t
$$
as ${n\to \infty}$.
\end{proof}

\begin{lemma}\label{lemma-integral}
Let $\Delta = [0,a] \subset [0,1]$ be a subinterval. Define for
every $\tau \in \Delta$ the $\Delta$-circling shift operator
$T_\tau\dopu  (T_\tau f)(t) = f(t)$ for $t > a$, $(T_\tau f)(t)
= f(t+\tau)$ for $0 \leq t \leq a - \tau$, and $(T_\tau f)(t) =
f(t - a + \tau)$ for $ a -\tau < t \leq a$. Then for every $f
\in X$ the map $\tau \mapsto T_\tau f$ is continuous in the
norm topology of $X$ and hence is Riemann integrable. Moreover,
$$
\frac{1}{a}\int_0^a T_\tau f \,d\tau =
\left(\frac{1}{a}\int_0^a f(t) \,dt\right)\ein_\Delta
+ f \ein_{[0,1] \setminus \Delta}.
$$
\end{lemma}

\begin{proof}
The fact is evident when $f$ is continuous and fulfills $f(0) =
f(a)$. Since, as remarked above,
such functions form a dense subset of $X$, we are
done.
\end{proof}

\begin{corollary}\label{cor-integral}
Let $[0,1]$ be split into a disjoint union of measurable
subsets $\Delta_1$ and $\Delta_2$. Then for every $g \in X$
$$
\left\||g|\ein_{\Delta_1} +
 \left(\frac{1}{\mu(\Delta_2)}\int_{\Delta_2} g(t) \,dt\right)\ein_{\Delta_2}\right\|_X
\leq \|g\|_X.
$$
\end{corollary}

\begin{proof}
We can assume without loss of generality  $\Delta_2 = [0, a]$ and apply
the previous lemma.
\end{proof}

\begin{corollary}\label{cor-integral+}
Let $g \in X$. Then for every $t \geq \mu(\supp g)$
$$
\frac{1}{t}\phi(t) \|g\|_1 \leq \|g\|_X
$$
\end{corollary}

\begin{proof}
We may assume without loss of generality that
$\Delta_2 := [0, t] \supset \supp g$ and apply
the previous corollary.
\end{proof}

\begin{lemma}\label{lemma-L1-converg}
Let $g \in L_\infty[0,1]$. Then for every $\alpha > 0$
$$
\|g\|_X \leq \alpha +  \|g\|_\infty \phi(\alpha^{-1} \|g\|_1).
$$
In particular, if $f_n \in L_\infty[0,1]$, $\sup_n \|f_n\|_\infty < \infty$
and $ \lim_{n \to \infty}\|f_n\|_1 = 0$, then $ \lim_{n \to \infty}\|f_n\|_X = 0$.
\end{lemma}

\begin{proof}
Remark that
$$
|g| \leq \alpha + \|g\|_\infty \ein_{\{\tau \in [0,1] \dopu  |g(\tau)| > \alpha\}},
$$
and that $\mu(\{\tau \in [0,1] \dopu  |g(\tau)| > \alpha\}) \leq
\alpha^{-1} \|g\|_1$.
\end{proof}

\begin{theorem}\label{thm-L1-Daug}
Let $X$ be a separable real  r.i.\  space
on $[0,1]$ with the following property: for every $\eps > 0$
there is an $f = f_{\eps} \in X$ such that
\begin{enumerate}
\item[(a)] $\|f\|_X =1$
\item[(b)] $\int_0^1f(t)\,dt < -1 + \eps$
\item[(c)]  $\|f + \ein\|_X \geq 2 - \eps$.
\end{enumerate}
Then $X = L_1[0,1]$ (endowed with its canonical norm).
\end{theorem}

Before giving the proof, we first record the main result of this
subsection as an immediate consequence.

\begin{corollary}\label{cor-L1-Daug}
The only separable real r.i.\ function space on $[0,1]$ with
the Daugavet property is $L_1[0,1]$ in its canonical norm.
\end{corollary}

Indeed, the characterisation of the Daugavet property in terms of slices
(\cite[Lemma~2.2]{KadSSW}) that was
cited in Section~2 allows us to
deduce this corollary from Theorem~\ref{thm-L1-Daug} by putting
$x = \ein$, $x^* = -\ein$ and taking as $f$ the corresponding $y$.

\begin{proof}[Proof of Theorem~\ref{thm-L1-Daug}]
Fix $\eps > 0$ and $f = f_{\eps} \in X$ with the properties
(a), (b) and (c). Consider the following partition:
$$
[0,1] = A \cup B = A_1 \cup A_2 \cup B_1 \cup B_2,
$$
where
\begin{align*}
A &= \{t \in [0,1]\dopu  f(t) \leq 0\}, & B &= \{t \in [0,1]\dopu
f(t) > 0\}, \\
A_1 &= \{t \in A\dopu  |f(t)| \leq 2\}, & A_2 &= \{t \in A\dopu  |f(t)| > 2\}, \\
B_1 &= \{t \in B\dopu  |f(t)| \leq 2\}, & B_2 &= \{t \in B\dopu  |f(t)| > 2\}
\end{align*}
(all these sets depend on $\eps$).

Remark first that (a) and (b) imply
\begin{equation} \label{d1}
\int_B f \,d\mu < \eps,
\end{equation}
otherwise $g = f \ein_A - f \ein_B$ would be a norm-one function with
$|{\int_0^1g(t)\,dt}| > 1$. In particular
$$
\int_{B_1} f \,d\mu < \eps,
$$
and Lemma~\ref{lemma-L1-converg} says  that
\begin{equation} \label{d2}
u(\eps):= \|f \ein_{B_1}\|_X \to 0
\end{equation}
as $\eps\to0$.
Since also $\int_{B_2} f \,d\mu < \eps$ and $f \geq 2$ on $B_2$ we have
\begin{equation} \label{d3}
\mu(B_2) < \frac{\eps}{2}.
\end{equation}
From  $\|f\|_1 \leq \|f\|_X =1$ we deduce
\begin{equation} \label{d4}
\mu(A_2 \cup B_2) < \frac{1}{2}.
\end{equation}
Now, using the facts
$|\ein_{A_1}+f\ein_{A_1}|\leq\ein_{A_1}$ and
$|\ein_{A_2}+f\ein_{A_2}|\leq |f|\ein_{A_2}$,
it is easy to check that
\begin{align*}
|\ein+f|\leq\ein_{A_1} + |f|\ein_{A_2} + \ein_{B_1}
+ |f|\ein_{B_2}+|f|\ein_{B_1}+\ein_{B_2}
\end{align*}
and, therefore, one can write
\begin{align*}
2 - \eps \leq \|\ein + f\| &
\leq \left\| \ein_{A_1} + |f|\ein_{A_2} + \ein_{B_1}
+ |f|\ein_{B_2}\right\|+ \left\||f|\ein_{B_1}\right\| +
\left\|\ein_{B_2}\right\|  \\
&
 \leq  \left\||f|\ein_{A_2\cup B_2} + \ein_{A_1\cup B_1}\right\|+ u(\eps) + \phi(\eps)
\end{align*}
by (\ref{d2}) and (\ref{d3}).
An application of Corollary~\ref{cor-integral} with $\Delta_1= A_2 \cup B_2$,
$\Delta_2= A_1 \cup B_1$ and
$$
g = |f| + \left(1 - \frac{1}{\mu(\Delta_2)} \int_{\Delta_2} |f|
 \,d\mu \right)\ein_{\Delta_2}
$$
implies
\begin{align}
2 - \eps &\leq \|g\|_X + u(\eps) + \phi(\eps) \nonumber \\
&\leq 1
+ \phi(\mu(\Delta_2))\left(1 - \frac{1}{\mu(\Delta_2)}
\int_{\Delta_2} |f| \,d\mu \right) + u(\eps) + \phi(\eps).
\label{d7}
\end{align}
Since we have by (\ref{d4}) $\mu(\Delta_2) \geq \frac12$ for all values
of $\eps$,
the last inequality implies $\lim_{\eps \to 0}\int_{\Delta_2} |f| \,d\mu = 0$.
Together with (\ref{d1}) this means that
\begin{equation} \label{d8}
 \lim_{\eps \to 0} \int_{A_2} |f| \,d \mu =1.
\end{equation}
Condition (\ref{d7}) also implies that
\begin{equation} \label{d9}
\lim_{\eps \to 0} \phi(\mu(\Delta_2)) =1.
\end{equation}
Since $\mu(A_2) \leq \mu(\Delta_1) \leq \mu(\Delta_2)$
we can apply Corollary~\ref{cor-integral+} for
$g = |f|\ein_{A_2}$ and $t =  \mu(\Delta_2)$. Then,
\begin{equation} \label{d10}
 1 \geq \frac{1}{\mu(\Delta_2)}\phi(\mu(\Delta_2)) \int_{A_2} |f| \,d \mu.
\end{equation}
By (\ref{d8}), (\ref{d9}) and (\ref{d10}) this implies $\mu(\Delta_2) \to 1$
and consequently $\mu(A_2) \to0$ as ${\eps \to 0} $.
Now we can apply again the same
Corollary~\ref{cor-integral+} but for  $t =  \mu(A_2)$ and $g = |f|\ein_{A_2}$.
This gives us that  $\liminf_{\eps \to 0} \phi(t)/t = 1$, and since $t \to 0$
as $\eps \to 0$ an application of Lemma~\ref{lemma-L1-char} completes the proof.
\end{proof}

\begin{remark}
Theorem~\ref{thm-L1-Daug} also implies that $L_1[0,1]$
is the only separable real r.i.\ space  on $[0,1]$ with ``bad
projections'' (see the definition in \cite{IvKa}) and the only
separable real
r.i.\  space on $[0,1]$ with the property that
$\|\Id + T\| = \|\Id - T\|$ for every rank-one operator $T$ (the
last property appeared in \cite{KMM3}).
This is so since the latter property is
stronger than the former one, and since the characterisation of
spaces with ``bad projections'' in terms of slices fits the
conditions of Theorem~\ref{thm-L1-Daug}.
\end{remark}

\begin{remark}
In order to extend the results in this subsection to the complex case,
one would have to replace (b)
of Theorem~\ref{thm-L1-Daug}
by $\int_0^1 \re f(t)\,dt < -1 + \eps$. Unfortunately we haven't
succeeded in proving this.
\end{remark}

\subsection{The  almost Daugavet property  in separable r.i.\ function spaces}

In this subsection we will deal with another weakening of the
\DP. Let $Y$ be a closed linear subspace of $X^*$. According to
\cite{KaShWer}, $X$  has the
\DP\ with respect to $Y$ if the Daugavet equation (\ref{0-eq1})
holds true for every rank-one operator $T\dopu  X \to X$ of the
form $T = f \otimes e$, where $e \in X$ and $f \in Y$. $X$ is
said to be an {\it almost Daugavet space} if there is a norming
subspace $Y \subset X^*$ such that $X$  has the \DP\ with
respect to $Y$. A separable space is an almost Daugavet space
\cite[Theorem~1.1]{KaShWer} if and only if
there is a sequence $(v_n) \subset B_X$ such that for every $x
\in X$
$$
\lim_{n\to\infty} \|x+v_n\| =\|x\|+1.
$$

 We will now show that there are r.i.\ renormings $X$ of $L_1[0,1]$ with the
 almost \DP; in fact,
$ \dist(X, L_1[0,1])$ (the Banach-Mazur distance from $X$ to
$L_1[0,1]$)
can be arbitrarily large.

\begin{theorem}\label{thm-L1-alpha}
For every $\alpha \in (0,1)$ denote $X_\alpha$ the linear space
$L_1[0,1]$ equipped with the norm given by
$$
p_\alpha(f) = \frac{1}{\alpha}\sup\biggl\{\int_A |f| \,d \mu \dopu  A
\in \Sigma, \ \mu(A) \leq \alpha\biggr\}\qquad (f\in X).
$$
Then, the following hold:
\begin{enumerate}
\item $X_\alpha$ is an almost Daugavet r.i.\ space;
\item $\dist(X_\alpha, L_1[0,1]) \to \infty$
as $\alpha \to 0$.
\end{enumerate}
\end{theorem}

\begin{proof}
By construction, $X_\alpha$ is rearrangement invariant.
Denote $v_n = \frac{\alpha}{n}\ein_{[0,1/n]}$. Evidently,
$p_\alpha(v_n) = 1$ for all $n >  \frac{1}{\alpha}$.
If we show that $\lim_{n\to\infty}p_\alpha(f + v_n) = p_\alpha(f) + 1$
for every $f \in X_\alpha$, then the almost
Daugavet property of $X_\alpha$ will be proved. Indeed, fix an $f \in X_\alpha$.
By the definition of $p_\alpha$ there is a sequence of $A_n \in
\Sigma$ such that $\mu(A_n) \leq \alpha$
and $\frac{1}{\alpha}\int_{A_n} |f| \,d \mu \to p_\alpha(f)$. By the
absolute continuity of the
Lebesgue integral one can modify $A_n$ in order to fulfill
additionally the
conditions  $\mu(A_n) \leq \alpha - \frac1n$, $A_n \cap [0,1/n] =
\emptyset$. Then
\begin{align*}
p_\alpha(f + v_n) &
\geq  \frac{1}{\alpha} \int_{A_n \cup [0,1/n]} |f + v_n| \,d\mu \\
&=  \frac{1}{\alpha} \int_{A_n} |f| \,d\mu +
\frac{1}{\alpha} \int_{[0,1/n]} |f + v_n| \,d\mu
\\
&\geq \frac{1}{\alpha} \int_{A_n} |f| \,d\mu
+ 1 - \frac{1}{\alpha} \int_{[0,1/n]} |f| \,d\mu \\
&\to  p_\alpha(f) + 1.
\end{align*}
So, (1) is proved. To prove (2) it is enough to remark that $X_\alpha$
contains a subspace isometric to $\ell_\infty^{(m)}$, where $m$ is the
entire part of
$1/\alpha$. This subspace is spanned by the functions
$\ein_{[0, \alpha]}, \ein_{[\alpha, 2 \alpha]}, \ldots , \ein_{[(m-1)\alpha, m \alpha]}$.
\end{proof}


\end{document}